\documentclass[reqno,12pt]{amsart}
\usepackage[letterpaper, margin=1in]{geometry}
\RequirePackage{amsmath,amssymb,amsthm,graphicx,mathrsfs,url,slashed,subcaption}
\RequirePackage[usenames,dvipsnames]{xcolor}
\RequirePackage[colorlinks=true,linkcolor=Red,citecolor=Green]{hyperref}
\RequirePackage{amsxtra}
\usepackage{cancel}
\usepackage{tikz-cd}
\usepackage{xcolor}
\usepackage{soul}
\usepackage{setspace}

\title{Computable categoricity relative to a c.e.\ degree}
\author{Java Darleen Villano}
\date{\today}
\address{Department of Mathematics\\ University of Connecticut\\ Storrs, Connecticut 06269}
\email{javavill@uconn.edu}

\newcommand{\NN}{\mathbb{N}}

\newcommand{\A}{\mathcal{A}}
\newcommand{\G}{\mathcal{G}}
\newcommand{\B}{\mathcal{B}}

\renewcommand{\epsilon}{\varepsilon}
\newcommand{\<}{\langle}
\renewcommand{\>}{\rangle}

\newcommand{\lessT}{<_{\text{T}}}

\renewcommand{\phi}{\varphi}

\newcommand{\use}{\mathrm{use}}
\newcommand{\age}{\mathrm{age}}

\newtheorem{theorem}{Theorem}[section]

\newtheorem{lemma}[theorem]{Lemma}

\newtheorem{question}[theorem]{Question}

\theoremstyle{definition}
\newtheorem{definition}[theorem]{Definition}



\numberwithin{equation}{section}

\def\XXint#1#2#3{{\setbox0=\hbox{$#1{#2#3}{\int}$ }
\vcenter{\hbox{$#2#3$ }}\kern-.6\wd0}}

\usepackage[backend=bibtex,style=alphabetic,giveninits=true]{biblatex}

\renewbibmacro{in:}{}
\DeclareFieldFormat{pages}{#1}

\begin{document}
\begin{abstract}
A computable graph $\G$ is computably categorical relative to a degree $\mathbf{d}$ if and only if for all $\mathbf{d}$-computable copies $\B$ of $\G$, there is a $\mathbf{d}$-computable isomorphism $f:\G\to\B$. In this paper, we prove that for every computable partially ordered set $P$ and computable partition $P=P_0\sqcup P_1$, there exists a computable computably categorical graph $\G$ and an embedding $h$ of $P$ into the c.e.\ degrees where $\G$ is computably categorical relative to all degrees in $h(P_0)$ and not computably categorical relative to any degree in $h(P_1)$. This is a generalization of a result by Downey, Harrison-Trainor, and Melnikov in \cite{MR4291596}.

\end{abstract}
% \thanks{I would like to thank my adviser David Reed Solomon for suggesting this project and providing many helpful comments as I wrote this paper. This research was partially supported by a Focused Research Group grant from the National Science Foundation of the United States, DMS-1854355.}
\maketitle

%%%%%%%%%%%%%%%%%%%%%%%%%%%%%%%%%%%%%%%%%%%%%%%%%%%%%%%

% \tableofcontents

\section{Introduction}
In computable structure theory, we are interested in effectivizing model theoretic notions and constructions. For a general background on computable structure theory, see Ash and Knight \cite{Ash2000-ASHCSA} or Montalb{\' a}n \cite{Montalban2021}. In particular, many people have examined the complexity of isomorphisms between structures within the same isomorphism type. We restrict ourselves to countable structures in a computable language and assume their domain is $\omega$. 

A computable structure $\mathcal{A}$ is \textbf{computably categorical} if for any computable copy $\mathcal{B}$ of $\mathcal{A}$, there exists a computable isomorphism between $\mathcal{A}$ and $\mathcal{B}$. There are many known examples of computably categorical structures including computable linear orderings with only finitely many adjacent pairs \cite{10.2307/2043534}, computable fields of finite transcendence degree \cite{https://doi.org/10.1002/malq.19770231902}, and computable ordered groups of finite rank \cite{GONCHAROV2003102}. In each of these examples, the condition given turns out to be both necessary and sufficient for computable categoricity. There are also examples of computable structures which are not computably categorical, such as $(\NN,\leq)$ as a linear order.

We are also interested in studying relativizations of computable categoricity. The most studied relativization of this notion is relative computable categoricity. A computable structure $\mathcal{A}$ is \textbf{relatively computably categorical} if for any copy $\mathcal{B}$ of $\mathcal{A}$, there exists a $\mathcal{B}$-computable isomorphism between $\mathcal{A}$ and $\mathcal{B}$. We can think of this relativization as relativizing categoricity to \textit{all} degrees at once since we do not fix the complexity of our copies of $\mathcal{A}$. If a structure $\mathcal{A}$ is relatively computably categorical, then it is computably categorical. The converse does not hold in general, but often holds for structures where there is a purely algebraic characterization of computable categoricity. In particular, the examples of computable categorical structures listed above are also relatively computably categorical.

The connection between a purely algebraic characterization of computable categoricity and the equivalence of computable categoricity and relative computable categoricity was clarified by the following result which was independently discovered by Ash, Knight, Manasse, and Slaman \cite{ASH1989195}, and Chisholm \cite{Chisholm90}.

\begin{theorem}[Ash et al., Chisholm]\label{formal c.e. Scott family}
    A structure $\mathcal{A}$ is relatively computably categorical if and only if it has a formally $\Sigma_1^0$ Scott family.
\end{theorem}

In this paper, we study the following relativization of categoricity, where for a computable structure $\mathcal{A}$, we consider noncomputable copies of $\mathcal{A}$ and ask for isomorphisms between $\mathcal{A}$ and those copies to have the same complexity as the copies.

\begin{definition}
    Let $X\in 2^\NN$. A structure $\mathcal{A}$ is \textbf{computably categorical relative to} $\mathbf{X}$ if and only if for all $X$-computable copies $\mathcal{B}$ of $\mathcal{A}$, there is an $X$-computable isomorphism $f:\mathcal{A}\to\mathcal{B}$. 
\end{definition}

We have that a computable structure $\mathcal{A}$ is relatively computably categorical when it is computably categorical relative to all $X\in2^\NN$. Given a computable structure $\mathcal{A}$, we can consider the set of $X\in2^\NN$ such that $\mathcal{A}$ is computably categorical relative to $X$ (or not). By Martin's result in \cite{martin68}, we obtain that either this set contains a cone in the Turing degrees or is disjoint from one. This motivates us to include the following definition which appears in a paper by Csima and Harrison-Trainor \cite{csima17}.

\begin{definition}\label{definition: categorical on a cone above a degree}
    A structure $\mathcal{A}$ is \textbf{computably categorical on a cone above} $\mathbf{d}$ if for all $\mathbf{c}\geq\mathbf{d}$, whenever $\mathcal{B}$ and $\mathcal{C}$ are $\mathbf{c}$-computable copies of $\mathcal{A}$, there is a $\mathbf{c}$-computable isomorphism between $\mathcal{B}$ and $\mathcal{C}$.
\end{definition}

We can think of $\mathcal{A}$ being relatively computably categorical as it being computably categorical on a cone above $\mathbf{0}$. We also have that a structure $\mathcal{A}$ is computably categorical on a cone above $\mathbf{d}$ if and only if it is computably categorical relative to $\mathbf{c}$ for every $\mathbf{c}\geq\mathbf{d}$. Downey, Harrison-Trainor, and Melnikov in \cite{MR4291596} showed that for every computable structure $\mathcal{A}$, the behavior of being categorical relative to a degree stabilizes on a cone above $\mathbf{0}''$. That is, either $\mathcal{A}$ is computably categorical relative to all degrees above $\mathbf{0}''$, or to no degree above $\mathbf{0}''$. Moreover, if a computable structure $\mathcal{A}$ is computably categorical on a cone above any degree $\mathbf{d}$, then it is also computably categorical on a cone above $\mathbf{d}\oplus\mathbf{0}''$. Since $\mathbf{d}\oplus\mathbf{0}''\geq\mathbf{0}''$, by the previous fact, we get that $\mathcal{A}$ is computably categorical on a cone above $\mathbf{0}''$. Therefore, by Martin's result, the set of degrees $\mathbf{d}$ such that $\mathcal{A}$ is computably categorical relative to $\mathbf{d}$ either contains the cone above $\mathbf{0}''$ or does not contain any cone at all.

The situation below $\mathbf{0}'$ was shown to be vastly different by the authors in \cite{MR4291596} proving that being computably categorical relative to a degree is not a monotonic property of a structure in the following way. 

\begin{theorem}[Downey, Harrison-Trainor, Melnikov]\label{DHM result}
    There is a computable structure $\A$ and c.e. degrees $0=Y_0\lessT X_0\lessT Y_1\lessT X_1\lessT\dots$ such that
    \begin{itemize}
        \item[(1)] $\A$ is computably categorical relative to $Y_i$ for each $i$,
        \item[(2)] $\A$ is not computably categorical relative to $X_i$ for each $i$,
        \item[(3)] $\A$ is relatively computably categorical to $\mathbf{0'}$.
    \end{itemize}
\end{theorem}

In this paper, we extend this result to partial orders of c.e.\ degrees, exhibiting that for the particular structure we build, the full extent in which it can change its categorical behavior relative to some c.e.\ degrees.

\begin{theorem}\label{main theorem}
    Let $P=(P,\leq)$ be a computable partially ordered set and let $P=P_0\sqcup P_1$ be a computable partition. Then, there exists a computable computably categorical directed graph $\G$ and an embedding $h$ of $P$ into the c.e.\ degrees where $\G$ is computably categorical relative to each degree in $h(P_0)$ and is not computably categorical relative to each degree in $h(P_1)$.
\end{theorem}

The proof is a priority construction on a tree of strategies, using several key ideas from the proof of Theorem \ref{DHM result} in \cite{MR4291596} along with some new techniques. In Section \ref{section 2}, we introduce informal descriptions for the strategies we need to satisfy our requirements for the construction and discuss important interactions between certain strategies. In Section \ref{section 3}, we detail the formal strategies, state and prove auxiliary lemmas about our construction, and state and prove the main verification lemma.

\subsection{Acknowledgements}
I would like to thank my adviser Reed Solomon for suggesting this project and providing many helpful comments as I wrote this paper. I would also like to thank the anonymous referee for their valuable comments and suggestions for improvement.

This research was partially supported by a Focused Research Group grant from the National Science Foundation of the United States, DMS-1854355.

\section{Informal strategies}\label{section 2}

To prove Theorem \ref{main theorem}, we have four goals to achieve within our construction, giving us four types of requirements to satisfy. In this section, we give informal descriptions of the strategies needed to satisfy each requirement in isolation, and then describe the interactions which arise when we employ these strategies together.

\subsection{Embedding $P$ into the c.e.\ degrees}\label{informal N strategy}
We embed the poset $P$ into the c.e.\ degrees in a standard way by constructing an independent family of uniformly c.e.\ sets $A_p$ for $p\in P$. We fix the following notation:
    \[
    \widehat{D_p} := \bigoplus_{q\neq p} A_q.
    \]

For each $p\in P$, we ensure that $A_p\not\leq_T\widehat{D_p}$. The image of $p$ will be the c.e.\ set $D_p=\bigoplus\limits_{q\leq p} A_q$. Because the $A_p$ are independent, our embedding is order-preserving, i.e., $p\leq q$ in $P$ if and only if $D_p\leq_T D_q$.

For each $p\in P$ and $e\in\omega$, we define the \textbf{independence requirement}:
    \[
    N_e^p : \Phi_e^{\widehat{D_p}} \neq A_p.
    \]

In order to satisfy an $N_e^p$ requirement in isolation, we use the following $N_e^p$-strategy. Let $\alpha$ be an $N_e^p$-strategy. When $\alpha$ is first eligible to act, it picks a large number $x_\alpha$.
Once $x_\alpha$ is defined, $\alpha$ checks if $\Phi_e^{\widehat{D_p}}(x_\alpha)[s]\downarrow=0$. If not, $\alpha$ takes no action at stage $s$.
If $\Phi_e^{\widehat{D_p}}(x_\alpha)[s]\downarrow=0$, then $\alpha$ enumerates $x_\alpha$ into $A_p$ and preserves this computation by restraining $\widehat{D_p}\restriction(\use(\Phi_e^{\widehat{D_p}}(x_\alpha))+1)$.

Notice that if we never see that $\Phi_e^{\widehat{D_p}}(x_\alpha)\downarrow=0$, then either $\Phi_e^{\widehat{D_p}}(x_\alpha)\uparrow$ or $\Phi_e^{\widehat{D_p}}(x_\alpha)\downarrow\neq 0$, and in either case, the value of $\Phi_e^{\widehat{D_p}}(x_\alpha)$ will not be equal to $A_p(x_\alpha)=0$ and so we meet the $N_e^p$ requirement. Otherwise, at the first stage $s$ for which $\Phi_e^{\widehat{D_p}}(x_\alpha)[s]\downarrow=0$, we enumerate $x_\alpha$ into $A_p$ and restrain $\widehat{D_p}$ below $\use(\Phi_e^{\widehat{D_p}}(x_\alpha))+1$. In this case we have that $\Phi_e^{\widehat{D_p}}(x_\alpha)\downarrow=0\neq 1=A_p(x_\alpha)$, and so we satisfy $N_e^p$.

\subsection{Making $\G$ computably categorical}\label{informal S strategy}
We will build $\G$ in stages. At stage $s=0$, we set $\G=\emptyset$. Then, at stage $s>0$, we add two new connected components to $\G[s]$ by adding the root nodes $a_{2s}$ and $a_{2s+1}$ for those components, and attaching to each node a $2$-loop (a cycle of length $2$). We then attach a $(5s+1)$-loop to $a_{2s}$ and a $(5s+2)$-loop to $a_{2s+1}$. This gives us the configuration of loops:
    \begin{align*}
        a_{2s} &: 2, 5s+1  \\
        a_{2s+1} &: 2, 5s+2.
    \end{align*}
The connected component consisting of the root node $a_{2s}$ with its attached loops will be referred to as the $\mathbf{2}s$\textbf{th connected component} of $\G$. During the construction, we might add more loops to connected components of $\G$, which causes them to have the following configuration:
    \begin{align*}
        a_{2s} &: 2, 5s+1, 5s+3  \\
        a_{2s+1} &: 2, 5s+1, 5s+4
    \end{align*}
We might also later add more loops to this configuration to obtain the following configuration:
    \begin{align*}
        a_{2s} &: 2, 5s+1, 5s+2, 5s+3  \\
        a_{2s+1} &: 2, 5s+1, 5s+2, 5s+4.
    \end{align*}

The idea behind adding these loops is to uniquely identify each connected component of $\G$. In all configurations above, there is only one way to match the components in $\G$ with components in a computable graph in order to define an embedding.

To make $\G$ computably categorical, we attempt to build an embedding of $\G$ into each computable directed graph. For each index $e$, let $\mathcal{M}_e$ be the (partial) computable graph with domain $\omega$ such that $E(x,y)\iff\Phi_e(x,y)=1$ and $\neg E(x,y)\iff\Phi_e(x,y)=0$. If $\Phi_e$ is not total, then $\mathcal{M}_e$ will not be a computable graph, but we will attempt to embed $\G$ into $\mathcal{M}_e$ anyway since we cannot know whether $\Phi_e$ is total or not. So we have the following requirement for each $e\in\omega$.
    \[
    S_e : \text{if $\G\cong\mathcal{M}_e$, then there exists a computable isomorphism $f_e:\G\to\mathcal{M}_e$}
    \]
To satisfy each $S_e$ requirement in isolation, we have the following strategy. Let $\alpha$ be an $S_e$-strategy. When $\alpha$ is first eligible to act, it sets its parameter $n_\alpha=0$ and defines its current map $f_\alpha$ from $\G$ into $\mathcal{M}_e$ to be empty. For the rest of this description, let $n=n_\alpha$. This parameter will keep track of the connected components that $\alpha$ is attempting to match between $\G$ and $\mathcal{M}_e$, and will be incremented by $1$ only when we find copies of the $2n$th and $(2n+1)$st connected components of $\G$. Suppose the map $f_\alpha[s-1]$ matches up the $2m$th and $(2m+1)$st components of $\G[s-1]$ and $\mathcal{M}_e[s-1]$ for all $m<n$. At future stages, $\alpha$ checks whether $\mathcal{M}_e[s]$ contains isomorphic copies of the $2n$th and $(2n+1)$st components in $\G[s]$. If not, $\alpha$ takes no additional action at stage $s$ and retains the parameter $n$ and the map $f_\alpha$. If $\mathcal{M}_e[s]$ contains copies of these components, then $\alpha$ extends the map by matching those components in $\G[s]$ and $\mathcal{M}_e[s]$, and it increments the value of $n$ by $1$. Since $\G$ and $\mathcal{M}_e$ are computable graphs, $f_\alpha$ will be correct on these components if $\G\cong\mathcal{M}_e$. When this isolated $S_e$-strategy is enacted along with the other strategies for other types of requirements, it may be the case that the initial definition of $f_\alpha$ on a pair of components is not total on all cycles attached to each root node, but if $\G\cong\mathcal{M}_e$, we will have a systematic way of extending $f_\alpha$ on any new cycles that were added (see Section \ref{section: interactions between multiple strategies}).

If $\alpha$ finds copies of the $2n$th and $(2n+1)$st components of $\G$ for every $n$, then $f_\alpha$ will be a computable embedding of $\G$ into $\mathcal{M}_e$. Because of the form of $\G$, if $\mathcal{M}_e\cong\G$, then $f_\alpha$ will be a partial embedding which can be extended computably to a computable embedding on all of $\G$, satisfying the $S_e$ requirement. Otherwise, there exists some $n$ such that the $2n$th and $(2n+1)$st components of $\G$ were never matched, and so $\G$ and $\mathcal{M}_e$ cannot be isomorphic and so we trivially satisfy $S_e$.

\subsection{Being computably categorical relative to a degree}\label{informal T strategy}
In this construction, we want to define computations using a $D_p$-oracle that can be destroyed later by enumerating numbers into $A_p$. We achieve this by setting the use of the $D_p$-computation to be $\<u,p\>$. Enumerating $u$ into $A_p$ causes $\<u,p\>$ to enter $D_p$, destroying the associated computation. 

For each $p\in P_0$, we ensure $\G$ is computably categorical relative to $D_p$. Let $\mathcal{M}_i^{D_p}$ be the (partial) $D_p$-computable directed graph with domain $\omega$ and edge relation given by $\Phi_i^{D_p}$. Since each $D_p$ is c.e., we define the following terms to keep track of certain finite subgraphs which appear and remain throughout our construction.

\begin{definition}
    Let $C_0$ and $C_1$ be isomorphic finite distinct subgraphs of $\mathcal{M}_i^{D_p}[s]$. The \textbf{age of} $\mathbf{C_0}$ is the least stage $t\leq s$ such that all edges in $C_0$ appear in $\mathcal{M}_i^{D_p}[t]$, denoted by $\age(C_0)$. We say that $\mathbf{C_0}$ is \textbf{older than} $\mathbf{C_1}$ when $\age(C_0)\leq\age(C_1)$.

    We say that $C_0$ is the \textbf{oldest} if for all finite distinct subgraphs $C\cong C_0$ of $\mathcal{M}_i^{D_p}[s]$, $\age(C_0)\leq\age(C)$.
\end{definition}

\begin{definition}
    Let $C_0=\<a_0,a_1,\dots,a_k\>$ and $C_1=\<b_0,b_1,\dots,b_k\>$ be isomorphic finite distinct subgraphs of $\mathcal{M}_i^{D_p}[s]$ with $a_0<a_1<\dots<a_k$ and $b_0<b_1<\dots<b_k$. We say that $C_0<_{\text{lex}} C_1$ if for the least $j$ such that $a_j\neq b_j$, $a_j<b_j$.

    We say that $C_0$ is the \textbf{lexicographically least} if for all finite distinct subgraphs $C\cong C_0$ of $\mathcal{M}_i^{D_p}[s]$, $C_0<_{\text{lex}} C$.
\end{definition}

If $\G\cong\mathcal{M}_i^{D_p}$, then we need to build a $D_p$-computable isomorphism between these graphs. To achieve this, we meet the following requirement for each $i\in\omega$.
    \[
    T_i^p : \text{if $\G\cong\mathcal{M}_i^{D_p}$, then there exists a $D_p$-computable isomorphism $g_i^{D_p}:\G\to\mathcal{M}_i^{D_p}$}
    \]

The strategy to satisfy each $T_i^p$ requirement in isolation is similar to the $S_e$-strategy, with some additional changes. Since the graphs are $D_p$-computable, embeddings defined by a $T_i^p$-strategy may become undefined later when small numbers enter $D_p$. Enumerations into $D_p$ also cause changes in $\mathcal{M}_i^{D_p}$ such as disappearing edges. We will show in the verification that if $\G\cong\mathcal{M}_i^{D_p}$, ``true'' copies of components from $\G$ will eventually appear and remain in $\mathcal{M}_i^{D_p}$ (and thus become the oldest finite subgraph which is isomorphic to a component in $\G$), and so our $T_i^p$-strategy below will be able to define the correct $D_p$-computable isomorphism between the two graphs. 

Let $\alpha$ be a $T_i^p$-strategy. When $\alpha$ is first eligible to act, it sets its parameter $n_\alpha=0$ and defines $g_\alpha^{D_p}$ to be the empty map. Once $\alpha$ has defined $n_\alpha$, then at the previous stage $s-1$ (or the last $\alpha$-stage in the full construction), we have the following situation: 
\begin{itemize}
    \item For each $m<n_\alpha$, $g_\alpha^{D_p}[s-1]$ maps the $2m$th and $(2m+1)$st components of $\G[s-1]$ to isomorphic copies in $\mathcal{M}_i^{D_p}[s-1]$.
    \item For $m<n_\alpha$, let $l_m$ be the maximum $\Phi_i^{D_p}[s-1]$-use for the loops in the copies in $\mathcal{M}_i^{D_p}[s-1]$ for the $2m$th and $(2m+1)$st components in $\G$. We can assume that if $m_0<m_1<n_\alpha$, then $l_{m_0}<l_{m_1}$.
    \item For $m<n_\alpha$, let $\<u_m,p\>$ be the $g_\alpha^{D_p}[s-1]$-use for the mapping of the $2m$th and $(2m+1)$st components of $\G$. This use will be constant for all elements in these components.
    \item By construction, we will have that $l_m<u_k\leq \<u_k,p\>$ for all $m\leq k<n_\alpha$.
\end{itemize}

Suppose $\alpha$ is acting at stage $s$ and has already defined $n_\alpha$. We first check whether numbers have been enumerated into $D_p$ that injure the loops in $\mathcal{M}_i^{D_p}[s-1]$ given by $\Phi_i^{D_p}[s-1]$; that is, if there exists a number $x<\use(\Phi_i^{D_p}[s-1])$ which entered $D_p$. If not, then we keep the current value of $n_\alpha$ and skip ahead to the next step. If so, then let $k$ be the least such that some number $x\leq l_k$ was enumerated into $D_p$. The loops in $\mathcal{M}_i^{D_p}[s]$ in the copies of the $2k$th and $(2k+1)$st components of $\G$ have been injured, and so may have disappeared. We have that $x\leq \<u_m,p\>$ for all $k\leq m<n_\alpha$, and so our map $g_\alpha^{D_p}[s]$ is now undefined on all the $2m$th and $(2m+1)$st components for $k\leq m<n_\alpha$. So, $\alpha$ redefines $n_\alpha=k$ to find new images for the $2k$th and $(2k+1)$st components in $\G[s]$.

Second, we check to see if there is a $j<k$ and a new $x\in D_p$ such that $l_j<x<\<u_j,p\>$. By minimality of the value $k$ above, we know that $l_j<x$ for all $j<k$ and $x\in D_p[s]\setminus D_p[s-1]$. For each such $j$, our map $g_\alpha^{D_p}[s-1]$ has been injured on the $2j$th and $(2j+1)$st components of $\G$, but the loops in the copies of those components in $\mathcal{M}_i^{D_p}[s]$ remain intact. Therefore, we define $g_\alpha^{D_p}[s]$ on these components with oracle $D_p[s]$ to be equal to $g_\alpha^{D_p}[s-1]$. Furthermore, we keep the same use for $g_\alpha^{D_p}[s]$ on these components. This will ensure that injury of this type happens only finitely often.

Third, we check whether we can extend $g_\alpha^{D_p}[s-1]$ to the $2n_\alpha$th and $(2n_\alpha+1)$st components of $\G[s]$. Search for isomorphic copies in $\mathcal{M}_i^{D_p}[s]$ of these components. If there are multiple copies in $\mathcal{M}_i^{D_p}[s]$, choose the oldest such copy to map to, and if there are multiple equally old copies, choose the lexicographically least oldest copy. If there are no copies in $\mathcal{M}_i^{D_p}[s]$, then keep the value of $n_\alpha$ the same and $g_\alpha^{D_p}$ unchanged and let the next requirement act. Otherwise, extend $g_\alpha^{D_p}[s-1]$ to $g_\alpha^{D_p}[s]$ to include the $2n_\alpha$th and $(2n_\alpha+1)$st components of $\G$ and set the use to be $\<u_{n_\alpha},p\>$ where $u_{n_\alpha}$ is large (and so $u_{n_\alpha}>l_k$ for all $k\leq n_\alpha)$. Increment $n_\alpha$ by $1$ and go to the next requirement.
    
If $\G\cong\mathcal{M}_i^{D_p}$, then for each $n$, eventually the real copies of the $2n$th and $(2n+1)$st components of $\G$ will appear and remain forever in $\mathcal{M}_i^{D_p}$. Moreover, they will eventually be the oldest and lexicographically least copies in $\mathcal{M}_i^{D_p}$. Let $l_n$ be the maximum true $D_p$-use on the edges in the loops in these $\mathcal{M}_i^{D_p}$ components. At this point, we will define $g_\alpha^{D_p}$ correctly on these components with a large use $\<u_n,p\>$. Since $l_n<\<u_n,p\>$, at most finitely many numbers enter $D_p$ from the interval $(l_n,\<u_n,p\>]$, but each time this happens, we define our map $g_\alpha^{D_p}$ to remain the same with the same use on the new oracle. Therefore, eventually our map $g_\alpha^{D_p}$ is never injured again on the $2n$th and $(2n+1)$st components. It follows that if $G\cong\mathcal{M}_i^{D_p}$, then $g_\alpha^{D_p}$ will be an embedding of $\G$ into $\mathcal{M}_i^{D_p}$ which will be an isomorphism by the structure of $\G$. 

\subsection{Being not computably categorical relative to a degree}\label{informal R strategy}
Finally, for each $q\in P_1$ we want to make $\G$ not computably categorical relative to the c.e.\ set $D_q$. To achieve this, we build a $D_q$-computable copy $\mathcal{B}_q$ of $\G$ such that for all $e\in\omega$, the $D_q$-computable map $\Phi_e^{D_q}:\G\to\B_q$ is not an isomorphism. The graph $\B_q$ will be built globally.

Similarly to $\G$, we build the directed graph $\B_q$ in stages. At stage $s=0$, we set $\B_q=\emptyset$. At stage $s>0$, we add root nodes $b_{2s}^q$ and $b_{2s+1}^q$ to $\B_q$ and attach to each one a $2$-loop. Next, we attach a $(5s+1)$-loop to $b_{2s}^q$ and a $(5s+2)$-loop to $b_{2s+1}^q$ with $D_q$-use $s$. However, throughout the construction, we may change the position of loops or add new loops to specific components of $\B_q$ depending on enumerations into $A_q$ (and thus into $D_q$). For the $2s$th and $(2s+1)$st components of $\B_q$, we have three possible final configurations of the loops. If we never start the process of diagonalizing using these components, then they will remain the same forever:
    \begin{align*}
        b_{2s}^q : 2, 5s+1  \\
        b_{2s+1}^q : 2, 5s+2. 
    \end{align*}
If we start, but don't finish, diagonalizing using these components, they will end in the following configuration:
    \begin{align*}
        b_{2s}^q : 2, 5s+1, 5s+3  \\
        b_{2s+1}^q : 2, 5s+1, 5s+4 
    \end{align*}
If we complete a diagonalization with these components, then they will end as:
    \begin{align*}
        b_{2s}^q : 2, 5s+1, 5s+2, 5s+4  \\
        b_{2s+1}^q : 2, 5s+1, 5s+2, 5s+3. 
    \end{align*}

For all $e\in\omega$, we meet the requirement
    \[
    R_e^q : \Phi_e^{D_q}:\G\to\mathcal{B}_q \ \text{is not an isomorphism}.
    \]
To satisfy this requirement, we will diagonalize against $\Phi_e^{D_q}$. Let $\alpha$ be an $R_e^q$-strategy.

When $\alpha$ is first eligible to act, it picks a large number $n_\alpha$, and for the rest of this strategy, let $n=n_\alpha$. This parameter indicates which connected components of $\B_q$ will be used in the diagonalization. At future stages, $\alpha$ checks if $\Phi_e^{D_q}$ maps the $2n$th and $(2n+1)$st connected component of $\G$ to the $2n$th and $(2n+1)$st connected component of $\B_q$, respectively. If not, $\alpha$ does not take any action. If $\alpha$ sees such a computation, it defines $m_\alpha$ to be the max of the uses of the computations on each component and restrains $D_q\restriction m_\alpha+1$.

At this point, our connected components in $\G[s]$ and $\B_q[s]$ are as follows:
    \begin{align*}
        a_{2n} : 2, 5n+1 & \ \ \ \ b_{2n}^q : 2, 5n+1 \\
        a_{2n+1} : 2, 5n+2 & \ \ \ \ b_{2n+1}^q : 2, 5n+2.
    \end{align*}
Since $\Phi_e^{D_q}$ looks like a potential isomorphism between $\G$ and $\B_q$, $\alpha$ will now take action to eventually force the true isomorphism to match $a_{2n}$ with $b^q_{2n+1}$ and to match $a_{2n+1}$ with $b^q_{2n}$ while preventing $\Phi_e^{D_q}$ from correcting itself on these components. Furthermore, $\alpha$ must do this in a way that will allow other requirements to succeed.

After $m_\alpha$ has been defined, $\alpha$ adds a $(5n+3)$-loop to $a_{2n}$ and a $(5n+4)$-loop to $a_{2n+1}$ in $\G[s]$. It also attaches a $(5n+3)$-loop to $b_{2n}^q$ and a $(5n+4)$-loop to $b_{2n+1}^q$ in $\B_q$[s]. Let $v_\alpha$ be a large unused number and set the use of all edges in these new loops appearing in $\B_q[s]$ to be $\<v_\alpha,q\>$. Notice that $\<v_\alpha,q\>>m_\alpha$ and that enumerating $v_\alpha$ into $A_q$ will put $\<v_\alpha,q\>$ into $D_q$, removing the $(5n+3)$- and $(5n+4)$-loops from $\B_q$ but not the $(5n+1)$- or $(5n+2)$-loops. Our connected components in $\G[s]$ and in $\B_q[s]$ are now:
    \begin{align*}
        a_{2n} : 2, 5n+1, 5n+3 & \ \ \ \ b_{2n}^q : 2, 5n+1, 5n+3 \\
        a_{2n+1} : 2, 5n+2, 5n+4 & \ \ \ \ b_{2n+1}^q : 2, 5n+2, 5n+4.
    \end{align*}

After adding the $(5n+3)$- and $(5n+4)$-loops to both graphs, $\alpha$ must now wait for higher priority strategies which have already defined their maps on the $2m$th and $(2m+1)$st components for all $m\leq n$ to recover their maps before taking the last step. If $\beta$ is a higher priority $S$ or $T^p$ strategy for which $f_\beta$ or $g_\beta^{D_p}$ is already defined on the $2n$th and $(2n+1)$st components of $\G$, then before completing its diagonalization, $\alpha$ must wait for $\beta$ to extend $f_\beta$ or $g_\beta^{D_p}$ to be defined on the new $(5n+3)$- and $(5n+4)$-loops. We will refer to this action as $\alpha$ issuing a challenge to all higher priority $S$ and $T$ requirements. Moreover, $\alpha$ issues a challenge to all higher priority $T$ requirements regardless of whether $p<q$, $q<p$, or if $p$ and $q$ are incomparable in $P$.

Once all higher priority strategies recover, $\alpha$ enumerates $v_\alpha$ into $A_q$ (and thus $\<v_\alpha,q\>$ goes into $D_q$). Doing this causes the $(5n+3)$-loop attached to $b_{2n}^q$ and the $(5n+4)$-loop attached to $b_{2n+1}^q$ to disappear in $\B_q[s]$. We now attach a $(5n+4)$-loop to $b_{2n}^q$ and a $(5n+3)$-loop to $b_{2n+1}^q$. We also attach a $(5n+1)$-loop to $a_{2n+1}$ and $b_{2n+1}^q$ and a $(5n+2)$-loop to $a_{2n}$ and $b_{2n}^q$, and we will refer to this process as homogenizing the components in $\G$ and in $\B_q$. The final configuration of our loops is:
    \begin{align*}
        a_{2n} : 2, 5n+1, 5n+2, 5n+3 & \ \ \ \ b_{2n}^q : 2, 5n+1, 5n+2, 5n+4 \\
        a_{2n+1} : 2, 5n+1, 5n+2, 5n+4 & \ \ \ \ b_{2n+1}^q : 2, 5n+1, 5n+2, 5n+3.
    \end{align*}

By homogenizing the components, we ensured that when we added loops for the diagonalization in $\B_q$, we also made adjustments in $\G$ to keep the components isomorphic to each other. Additionally, because $v_\alpha>m_\alpha$, the values $\Phi_e^{D_q}(a_{2n})=b_{2n}$ and $\Phi_e^{D_q}(a_{2n+1})=b_{2n+1}$ remain. So if $\Phi_e^{D_q}[s]$ is extended to a map on the entirety of $\G$, it cannot be a $D_q$-computable isomorphism, and so $R_e^q$ is satisfied. If we meet $R_e^q$ for all $e\in\omega$, we have that $\G$ is not computably categorical relative to $D_q$ with $\B_q$ being the witness.

\subsection{Interactions between multiple strategies}\label{section: interactions between multiple strategies}
There are some interactions which can cause problems between these strategies. We will explain how the strategies described in this section, with some tweaks, can solve these issues.

We first point out that the independence requirements cause no serious issues for the other requirements. An $N_e^p$-strategy $\alpha$, when it is first eligible to act, will pick a large unused number $x_\alpha$, and so if it ever enumerates $x_\alpha$ into $A_p$, it will not violate any restraints placed by higher priority independence or $R_e^q$ requirements. If this enumeration injures loops in or embeddings defined on components of $\mathcal{M}_i^{D_r}$ for some higher priority $T_i^r$-strategy where $p\leq r$, the $T_i^r$-strategy will be able to check for this $D_r$ change when it is next eligible to act and will be able to react accordingly to succeed.

The next main interaction to note is between an $R_e^q$-strategy $\beta$ and an $S$- or $T$-strategy $\alpha$ where $\alpha^\frown\<\infty\>\subseteq\beta$. In the tree of strategies, $\alpha^\frown\<\infty\>\subseteq\beta$ indicates that $\beta$ guesses $\alpha$ will define an embedding of $\G$ into its graph $\mathcal{M}_i$ or $\mathcal{M}_i^{D_p}$. In the informal description for $\beta$, we had $\beta$ wait for higher priority strategies to recover their embeddings defined on $\G$ after we added $(5n+3)$- and $(5n+4)$-loops to components of $\G$. When $\beta$ adds the new loops, it updates $\alpha$'s parameter by setting $n_\alpha=n_\beta$ if $\alpha$ is an $S$-strategy. If $\alpha$ is instead a $T$-strategy, then $\beta$ updates $n_\alpha$ to be the least $m\leq n_\beta$ such that $g_\alpha^{D_p}$ is no longer fully defined on the $2m$th and $(2m+1)$st components of $\G$ (for reasons that will become clear below). In either case, this causes $\alpha$ to return to previous components of $\G$ to find copies of them in either $\mathcal{M}_e$ or $\mathcal{M}_i^{D_p}$. If it is the case that either $\G\cong\mathcal{M}_e$ or $\G\cong\mathcal{M}_i^{D_p}$, $\alpha$ will eventually find copies and is able to define either a computable or $D_p$-computable isomorphism between the two graphs. Hence, the only tweaks needed for the $S_e$- and $T_i^p$-strategies $\alpha$ are steps in which they check if there is a lower priority $R_e^q$-strategy $\beta$ where $\alpha^\frown\<\infty\>\subseteq\beta$ that has issued its challenge after adding the new initial loops in $\G$. 

There is a related technical point concerning $R_e^q$ homogenizing the $2n_\beta$th and $(2n_\beta+1)$st components in the last step of its diagonalization. We do not ask the higher priority $S$- and $T$-strategies $\alpha$ with $\alpha^\frown\<\infty\>\subseteq\beta$ to go back and match these final homogenizing loops. Instead, we will extend $f_\alpha$ (or $g_\alpha^{D_p}$) to those homogenizing loops in a computable (or $D_p$-computable) way for $\alpha$ on the true path in the verification.

One last interaction arises between an $R_e^q$-strategy $\beta$ and a $T_i^p$-strategy $\alpha$ where $\alpha^\frown\<\infty\>\subseteq\beta$ and $q<p$ in $P$. With the current construction, we could have the following situation which makes it impossible for $\alpha$ to succeed. Suppose $\alpha$ finds copies at a stage $s_0$ of the $2n_\beta$th and $(2n_\beta+1)$st components of $\G[s_0]$ into $\mathcal{M}_i^{D_p}[s_0]$, and we have the following components in both graphs
    \begin{align*}
        a_{2n_\beta} : 2, 5n_\beta+1 & \ \ \ \ c : 2, 5n_\beta+1 \\
        a_{2n_\beta+1} : 2, 5n_\beta+2 & \ \ \ \ d : 2, 5n_\beta+2
    \end{align*}
where $c$ and $d$ are the root nodes of the copies of the $\G$ components found in $\mathcal{M}_i^{D_p}[s_0]$. At this point, $\alpha$ defines $g_\alpha^{D_p}[s_0]$ by mapping $a_{2n_\beta}\mapsto c$ and $a_{2n_\beta+1}\mapsto d$ with a use $\<u_{\alpha,n_\beta},p\>$. 

Suppose at a later stage $s_1>s_0$, $\beta$ adds new loops to the corresponding components in $\G$ with a $D_q$-use of $\<v_\alpha,q\>$ for $v_\alpha$ large and issues its challenge, and so our components in $\G[s_1]$ and $\mathcal{M}_i^{D_p}[s_1]$ are
    \begin{align*}
        a_{2n_\beta} : 2, 5n_\beta+1, 5n_\beta+3 & \ \ \ \ c : 2, 5n_\beta+1 \\
        a_{2n_\beta+1} : 2, 5n_\beta+2, 5n_\beta+4 & \ \ \ \ d : 2, 5n_\beta+2.
    \end{align*}
Then, suppose at stage $s_2>s_1$ that $\Phi_i^{D_p}[s_1]$ adds the new loops correctly to $c$ and $d$:
    \begin{align*}
        a_{2n_\beta} : 2, 5n_\beta+1, 5n_\beta+3 & \ \ \ \ c : 2, 5n_\beta+1, 5n_\beta+3 \\
        a_{2n_\beta+1} : 2, 5n_\beta+2, 5n_\beta+4 & \ \ \ \ d : 2, 5n_\beta+2, 5n_\beta+4.
    \end{align*}
Let $z_{\alpha,n_\beta}$ be the minimum use for any of the new edges in these loops, and assume that $\<v_\alpha,q\><z_{\alpha,n_\beta}$. The strategy $\alpha$ extends $g_\alpha^{D_p}[s_1]$ to map the $(5n_\beta+3)$- and $(5n_\beta+4)$-loops from $\G$ into $\mathcal{M}_i^{D_p}$ with a large use (i.e., greater than $\<u_{\alpha,n_\beta},p\>$). $\alpha$ has now met its challenge and takes the $\infty$ outcome.

Finally, suppose at a stage $s_3>s_2$, $\beta$ is eligible to act again. $\beta$ enumerates $v_\alpha$ into $A_q$, which enumerates $\<v_\alpha,q\>$ into $D_q$ and $D_p$ since $q<p$. Since $\<v_\alpha,q\>\in D_q$, the $(5n_\beta+3)$- and $(5n_\beta+4)$-loops in $\B_q$ disappear, and so $\beta$ can homogenize the $2n_\beta$th and $(2n_\beta+1)$st components in $\G$ and in $\B_q$ and diagonalize.
    \begin{align*}
        a_{2n_\beta} : 2, 5n_\beta+1, 5n_\beta+2, 5n_\beta+3 & \ \ \ \ b_{2n_\beta}^q : 2, 5n_\beta+1, 5n_\beta+2, 5n_\beta+4 \\
        a_{2n_\beta+1} : 2, 5n_\beta+1, 5n_\beta+2, 5n_\beta+4 & \ \ \ \ b_{2n_\beta+1}^q : 2, 5n_\beta+1, 5n_\beta+2, 5n_\beta+3.
    \end{align*}
The map $\Phi_e^{D_q}[s_3]$ still maps $a_{2n_\beta}$ to $b_{2n_\beta}^q$ and $a_{2n_\beta+1}$ to $b_{2n_\beta+1}^q$ since the $D_q$-use for these computations is less than $v_\alpha$ and thus less than $\<v_\alpha,q\>$.

However, because $\<v_\alpha,q\>$ has gone into $D_p$ as well and $\<v_\alpha,q\><z_{\alpha,n_\beta}$, the $(5n_\beta+3)$- and $(5n_\beta+4)$-loops in $\mathcal{M}_i^{D_p}[s_3]$ also disappear. But $v_\alpha$ was chosen to be large at stage $s_1$, so $v_\alpha>u_{\alpha,n_\beta}$ and so $g_\alpha^{D_p}[s_3]$ still maps $a_{2n_\beta}$ to $c$ and $a_{2n_\beta+1}$ to $d$. This allows the opponent controlling $\mathcal{M}_i^{D_q}$ to add loops in the following way to diagonalize:
    \begin{align*}
        a_{2n_\beta} : 2, 5n_\beta+1, 5n_\beta+2, 5n_\beta+3 & \ \ \ \ c : 2, 5n_\beta+1, 5n_\beta+2, 5n_\beta+4 \\
        a_{2n_\beta+1} : 2, 5n_\beta+1, 5n_\beta+2, 5n_\beta+4 & \ \ \ \ d : 2, 5n_\beta+1, 5n_\beta+2, 5n_\beta+3.
    \end{align*}
This now makes it impossible for $\alpha$ to succeed if $\G\cong\mathcal{M}_i^{D_p}$. 

To solve this conflict, $\alpha$ needs to lift the use of $g_\alpha^{D_p}[s_1]$ when $\beta$ starts its diagonalization process. Specifically, when $\beta$ adds the $(5n_\beta+3)$- and $(5n_\beta+4)$-loops to $\B_q$ and sets their $D_q$-use to be $\<v_\alpha,q\>$, we enumerate $u_{\alpha,n_\beta}$ into $A_p$, which puts $\<u_{\alpha,n_\beta},p\>$ into $D_p$ but nothing into $D_q$. This action makes $g_\alpha^{D_p}$ undefined on the $2n_\beta$th and $(2n_\beta+1)$st components of $\G$. When $\alpha$ is next eligible to act, it will redefine $g_\alpha^{D_p}$ on the $2n_\beta$th and $(2n_\beta+1)$st components of $\G$ with a large use greater than $\<v_\alpha,q\>$. Therefore, if $\beta$ later enumerates $v_\alpha$ into $A_q$ to diagonalize, the map $g_\alpha^{D_p}$ will become undefined on the entirety of the $2n_\beta$th and $(2n_\beta+1)$st components, preventing the opponent from using $\mathcal{M}_i^{D_p}$ to diagonalize against $\alpha$.

It is possible that there is more than one $T$-strategy $\alpha$ with $\alpha^\frown\<\infty\>\subseteq\beta$ associated with elements in $P$ greater than $q$. In this case, $\beta$ has to enumerate the use $u_{\alpha,n_\beta}$ for each such $\alpha$ into $A_q$. These elements may cause $g_\alpha^{D_p}$ to become undefined on the $2m$th and $(2m+1)$st components for $m<n_\beta$, or for these components in $\mathcal{M}_\alpha^{D_p}$ to disappear. Therefore, for a $T_i^p$-strategy $\alpha$ with $\alpha^\frown\<\infty\>\subseteq\beta$ and $q<p$, $\beta$ resets $n_\alpha$ to be the least $m\leq n_\beta$ such that $g_\alpha^{D_p}$ no longer matches the $2m$th and $(2m+1)$st components in $\G$ and $\mathcal{M}_i^{D_p}$.

\section{Proof of Theorem \ref{main theorem}}\label{section 3}

In this section, we prove Theorem \ref{main theorem}. Fix a computable partially ordered set $P=(P,\leq)$, and  let $P=P_0\sqcup P_1$ be a computable partition of $P$.

We build our computable directed graph $\G$ stage by stage as outlined in section \ref{informal S strategy}, and for each $q\in P_1$, we build an isomorphic copy $\B_q$ of $\G$ as outlined in section \ref{informal R strategy}.
    
We will also build a uniformly c.e.\ family of independent sets $A_p$ for $p\in P$ via a priority argument on a tree of strategies. We define
    \[
    D_p = \bigoplus_{q\leq p} A_q
    \]
and
    \[
    \widehat{D_p}=\bigoplus_{q\neq p} A_q.
    \]

Our embedding $h$ of $P$ into the c.e.\ degrees is the map $h(p)=D_p$ for all $p\in P$.

\subsection{Requirements}
Recall our four types of requirements for our construction:
    \[
    N_e^p : \Phi_e^{\widehat{D_p}} \neq A_p
    \]
    \[
    S_e : \text{if $\G\cong\mathcal{M}_e$, then there exists a computable isomorphism $f_e:\G\to\mathcal{M}_e$}
    \]
    \[
    T_i^p : \text{if $\G\cong\mathcal{M}_i^{D_p}$, then there exists a $D_p$-computable isomorphism $g_i^{D_p}:\G\to\mathcal{M}_i^{D_p}$}
    \]
    \[
    R_e^q : \Phi_e^{D_q}:\G\to\mathcal{B}_q \ \text{is not an isomorphism}
    \]

\subsection{Construction}
Let $\Lambda=\{\infty<_\Lambda\dots<_\Lambda w_2<_\Lambda s<_\Lambda w_1<_\Lambda w_0\}$ be the set of outcomes, and let $T=\Lambda^{<\omega}$ be our tree of strategies. The construction will be performed in $\omega$ many stages $s$.

We define the \textbf{current true path} $\pi_s$, the longest strategy eligible to act at stage $s$, inductively. For every $s$, $\lambda$, the empty string, is eligible to act at stage $s$. Suppose the strategy $\alpha$ is eligible to act at stage $s$. If $|\alpha|<s$, then follow the action of $\alpha$ to choose a successor $\alpha^\frown\<o\>$ on the current true path. If $|\alpha|=s$, then set $\pi_s=\alpha$. For all strategies $\beta$ such that $\pi_s <_L\beta$, initialize $\beta$ (i.e., set all parameters associated to $\beta$ to be undefined). If $\beta <_L \pi_s$ and $|\beta|<s$, then $\beta$ retains the same values for its parameters.

We will now give formal descriptions of each strategy and their outcomes in the construction.

\subsection{$N_e^p$-strategies and outcomes}
We first cover the $N_e^p$-strategies used to make each c.e.\ set $A_p$ independent. Let $\alpha$ be an $N_e^p$-strategy eligible to act at stage $s$.

\textbf{Case $\mathbf{1}$}: If $\alpha$ is acting for the first time at stage $s$ or has been initialized since the last $\alpha$-stage, define its parameter $x_\alpha$ to be large, and take outcome $w_0$.

\textbf{Case $\mathbf{2}$}: If $x_\alpha$ is already defined and $\alpha$ took outcome $w_0$ at the last $\alpha$-stage, check if 
    \[
    \Phi_e^{\widehat{D_p}}(x_\alpha)[s]\downarrow=0.
    \]
If not, take the $w_0$ outcome. If $\Phi_e^{\widehat{D_p}}(x_\alpha)[s]\downarrow=0$, enumerate $x_\alpha$ into $A_p$ and take the $s$ outcome which will preserve $\widehat{D_p}\restriction(\text{use}(\Phi_e^{\widehat{D_p}}(x_\alpha)[s])+1)$.

\textbf{Case $\mathbf{3}$}: If $\alpha$ took the $s$ outcome the last time it was eligible to act and has not been initialized, take the $s$ outcome again.

\subsection{$S_e$-strategies and outcomes}

We now detail our $S_e$-strategy to make $\G$ computably categorical. Let $\alpha$ be an $S_e$-strategy eligible to act at stage $s$.

\textbf{Case $\mathbf{1}$}: If $\alpha$ is acting for the first time or has been initialized since the last $\alpha$-stage, define $n_\alpha=0$ and $f_\alpha[s]$ to be the empty map. Take the $w_0$ outcome.

\textbf{Case $\mathbf{2}$}: If $\alpha$ has defined $n_\alpha$ and is currently challenged by an $R_e^q$-strategy $\beta$ with $\alpha^\frown\<\infty\>\subseteq\beta$, then $\alpha$ acts as follows. In the verification, we show that there can only be one strategy challenging $\alpha$ at a time. When $\beta$ challenged $\alpha$, if the value of $n_\alpha$ was greater than $n_\beta$, then $\beta$ redefined $n_\alpha$ to be equal to $n_\beta$. Therefore, we currently have $n_\alpha\leq n_\beta$. Furthermore, if $n_\alpha=n_\beta$, then $f_\alpha$ may already be defined on the $(5n_\alpha+1)$- and $(5n_\alpha+2)$-loops in the $2n_\alpha$th and $(2n_\alpha+1)$st components of $\G$. 

$\alpha$ searches for copies of the $2n_\alpha$th and $(2n_\alpha+1)$st components of $\G$ in $\mathcal{M}_e[s]$. More formally, if $f_\alpha$ is not defined on any loops in these components, then $\alpha$ searches for full copies of both components in $\mathcal{M}_e[s]$. If $n_\alpha=n_\beta$ and $f_\alpha$ is already defined on the $(5n_\alpha+1)$- and $(5n_\alpha+2)$-loops in $\G$, then $\alpha$ searches for the new loops of lengths $5n_\alpha+3$ and $5n_\alpha+4$ in the matched components in $\mathcal{M}_e[s]$. If no copies are found, set $f_\alpha[s]=f_\alpha[s-1]$, leave $n_\alpha$ unchanged, and take the $w_{n_\alpha}$ outcome. If copies are found, extend $f_\alpha[s-1]$ to $f_\alpha[s]$ by matching the components, increment $n_\alpha$ by $1$ and check if $n_\alpha>n_\beta$ for this new $n_\alpha$. If yes, take the $\infty$ outcome and declare $\beta$'s challenge to have been met. If not, take the $w_{n_\alpha}$ outcome and let $\beta$'s challenge remain active.

\textbf{Case $\mathbf{3}$}: If $\alpha$ has defined its parameter $n_\alpha$ and $\alpha$ is not currently challenged, then, $\alpha$ continues to search for copies of the $2n_\alpha$th and $(2n_\alpha+1)$st components of $\G$ in $\mathcal{M}_e[s]$. If no copies are found, define $f_\alpha[s]=f_\alpha[s-1]$, leave $n_\alpha$ unchanged, and take the $w_{n_\alpha}$ outcome. Otherwise, extend $f_\alpha[s-1]$ to $f_\alpha[s]$ by mapping the components to their respective copies in $\mathcal{M}_e[s]$, increment $n_\alpha$ by $1$, and take the $\infty$ outcome.

\subsection{$T_i^p$-strategies and outcomes}\label{formal T strategies}

For each $p\in P_0$, we have the following $T_i^p$-strategy. Let $\alpha$ be a $T_i^p$-strategy eligible to act at stage $s$.

\textbf{Case $\mathbf{1}$}: If $\alpha$ is acting for the first time or has been initialized since the last $\alpha$-stage, set $n_\alpha=0$, define $g_\alpha^{D_p}[s]$ to be the empty function, and take the $w_0$ outcome. 

\textbf{Case $\mathbf{2}$}: $\alpha$ is currently challenged by an $R_e^q$-strategy $\beta$ where $\alpha^\frown\<\infty\>\subseteq\beta$. Let $s_0$ be the stage at which $\beta$ challenged $\alpha$. In the verification, we will show that there can only be one strategy challenging $\alpha$ at a time. When $\beta$ challenged $\alpha$ at stage $s_0$, it redefined $n_\alpha$ to be equal to the least $m\leq n_\beta$ such that $g_\alpha^{D_p}$ is not fully defined on the $2m$th and $(2m+1)$st components of $\G$ (for example, we may have that $m<n_\beta$ if an enumeration by $\alpha$'s challenge injures loops in $\mathcal{M}_i^{D_p}$).

If $q<p$ and $s$ is the first $\alpha$-stage since $s_0$ and $n_\alpha$ was greater than $n_\beta$ at stage $s_0$, then we have to perform a preliminary action. In this case, $\beta$ enumerated $u_{\alpha,n_\beta}$ into $A_p$, causing the map $g_\alpha^{D_p}[s_0]$ to become undefined on the $2n_\beta$th and $(2n_\beta+1)$st components of $\G$. Choose a new large number $u_{\alpha,n_\beta}$ and redefine $g_\alpha^{D_p}[s]$ to be equal to $g_\alpha^{D_p}[s_0]$ on the $2$-loops, $(5n_\beta+1)$-loops, and $(5n_\beta+2)$-loops in these components with use $\<u_{\alpha,n_\beta},p\>$. These loops are still defined in $\mathcal{M}_i^{D_p}[s]$ since their $D_p$-uses remained the same and thus were below the old $u_{\alpha,n_\beta}$. This ends the preliminary step.

Next, we perform the main action in this case. If $n_\alpha=n_\beta$ and $\alpha$ is already defined on the $(5n_\alpha+1)$- and $(5n_\alpha+2)$-loops of the $2n_\alpha$th and $(2n_\alpha+1)$st components in $\G$, then $\alpha$ searches for the oldest and lexicographically least copies of the $(5n_\alpha+3)$- and $(5n_\alpha+4)$-loops in $\mathcal{M}_i^{D_p}[s]$. If $g_\alpha^{D_p}$ is not currently defined on any of the loops in the $2n_\alpha$th and $(2n_\alpha+1)$st components of $\G$, then $\alpha$ searches for the oldest and lexicographically least copies of these components in $\mathcal{M}_i^{D_p}[s]$. In either case, if such copies are found, extend $g_\alpha^{D_p}[s]$ to map onto these copies with use $\<u_{\alpha,n_\alpha},p\>$ for large $u_{\alpha,n_\alpha}$, increment $n_\alpha$ by $1$, and check if $n_\alpha>n_\beta$ for this new $n_\alpha$. If yes, take the $\infty$ outcome and declare $\beta$'s challenge to $\alpha$ to be met. If not, then take the $w_{n_\alpha}$ outcome and let $\beta$'s challenge to $\alpha$ remain active.

\textbf{Case $\mathbf{3}$}: $\alpha$ is not currently challenged by an $R_e^q$-strategy. Let $t$ be the last $\alpha$-stage. In this case, $\alpha$ defined $g_\alpha^{D_p}[t]$ on the $2m$th and $(2m+1)$st components with use $\<u_{\alpha,m},p\>$ for $m<n_\alpha$. Let $l_m$ be the max $D_p$-use for the computation of a loop in the image of the $2m$th and $(2m+1)$st components under $g_\alpha^{D_p}[t]$. In the verification, we will show that $l_m<u_{\alpha,m}$ for all $m<n_\alpha$.

\textbf{Step $\mathbf{1}$}: If there is an $m<n_\alpha$ such that $D_p[t]\restriction l_m\neq D_p[s]\restriction l_m$, then let $m$ be the least such value. Note that for $m\leq m^*<n_\alpha$, the map $g_\alpha^{D_p}$ is now undefined on the $2m^*$th and $(2m^*+1)$st components of $\G$. The loops in the image of the $2k$th and $(2k+1)$st components of $\G$ under $g_\alpha^{D_p}[t]$ for $k<m$ remain in $\mathcal{M}_i^{D_p}[s]$. Update $n_\alpha=m$.

\textbf{Step $\mathbf{2}$}: By the update in Step $1$, we have that for each $m<n_\alpha$ that $D_p[t]\restriction l_m=D_p[s]\restriction l_m$. For each $m<n_\alpha$, if any, where $D_p[t]\restriction \<u_{\alpha,m},p\>\neq D_p[s]\restriction \<u_{\alpha,m},p\>$, set $g_\alpha^{D_p}[s]=g_\alpha^{D_p}[t]$ on the loops in $\G$ in the $2m$th and $(2m+1)$st components with the same use as at stage $t$.

\textbf{Step $\mathbf{3}$}: We can now perform the main action of this case. $\alpha$ searches for the oldest and lexicographically least copies of the $2n_\alpha$th and $(2n_\alpha+1)$st components of $\G$ in $\mathcal{M}_i^{D_p}[s]$. If no copies are found, leave $g_\alpha^{D_p}$ and $n_\alpha$ unchanged and take outcome $w_{n_\alpha}$. Otherwise, extend $g_\alpha^{D_p}$ by mapping the loops in the $2n_\alpha$th and $(2n_\alpha+1)$st components of $\G$ to their copies in $\mathcal{M}_i^{D_p}$ with use $\<u_{\alpha,n_\alpha},p\>$ where $u_{\alpha,n_\alpha}$ is chosen large,  increment $n_\alpha$ by $1$, and take the $\infty$ outcome.

\subsection{$R_e^q$-strategies and outcomes}\label{R strategy}

Finally, for each $q\in P_1$, we have the following $R_e^q$-strategy. Let $\alpha$ be an $R_e^q$-strategy eligible to act at stage $s$. Recall that $\B_q$ is a $D_q$-computable copy of $\G$ which follows $\G$ on components not chosen by any $R$-strategy. For components chosen by $\alpha$ in particular, we have the following.

\textbf{Case $\mathbf{1}$}: If $\alpha$ is first eligible to act at stage $s$ or has been initialized, define the parameter $n_\alpha=n$ to be large and take outcome $w_0$.

\textbf{Case $\mathbf{2}$}: If we are not in Case $1$ and $\alpha$ took outcome $w_0$ at the last $\alpha$-stage, check whether $\Phi_e^{D_q}[s]$ maps the $2n$th and $(2n+1)$st components of $\G$ isomorphically into $\B_q$. If not, take outcome $w_0$.

If so, set $m_\alpha$ to be the maximum $D_q$-use of these computations. Let $v_\alpha$ be large. Add a $(5n+3)$-loop to $a_{2n}$ in $\G$ (computably) and to $b_{2n}^q$ in $\B_q$ (with $D_q$-use $\<v_\alpha,q\>$) and add a $(5n+4)$-loop to $a_{2n+1}$ in $\G$ (computably) and to $b_{2n+1}^q$ in $\B_q$ (with $D_q$-use $\<v_\alpha,q\>$). 

For each $T_i^p$-strategy $\gamma$ where $\gamma^\frown\<\infty\>\subseteq\alpha$ and $q<p$, enumerate the use $u_{\gamma,n}$ into $A_p$ (and so $\<u_{\gamma,n},p\>$ enters $D_p$), and challenge $\gamma$. Note that if $n_\gamma<n_\alpha$, then there is no $u_{\gamma,n}$ to enumerate into $A_p$. For each $S$-strategy $\beta$ where $\beta^\frown\<\infty\>\subseteq\alpha$, challenge $\beta$ and reset $n_\beta=n_\alpha$ if $n_\alpha>n_\beta$. Otherwise, leave $n_\alpha$ as it is. For each $T$-strategy $\beta$ where $\beta^\frown\<\infty\>\subseteq\alpha$, reset $n_\beta$ to be the least $m\leq n_\alpha$ such that $g_\alpha^{D_p}$ does not match all of the $2m$th and $(2m+1)$st components of $\G$. Take outcome $w_1$.

\textbf{Case $\mathbf{3}$}: If $\alpha$ took outcome $w_1$ at the last $\alpha$-stage, then enumerate $v_\alpha$ into $A_q$, move the $(5n+3)$-loop in $\B_q$ from $b_{2n}^q$ to $b_{2n+1}^q$, and move the $(5n+4)$-loop in from $b_{2n+1}^q$ to $b_{2n}^q$. Attach a $(5n+1)$-loop to $a_{2n+1}$ and $b_{2n+1}^q$ and a $(5n+2)$-loop to $a_{2n}$ and $b_{2n}^q$. Let the $D_q$-use of these new loops in $\B_q$ be equal to the stage number $s$. Take outcome $s$. 

\textbf{Case $\mathbf{4}$}: If $\alpha$ took outcome $s$ at the last $\alpha$-stage and has not been initialized, then take outcome $s$.

\subsection{Verification}

We first prove that the map $h$ where $h(p)=D_p$ is an embedding into the c.e.\ degrees (if all $N_e^p$ requirements are satisfied) and the computable and $D_p$-computable embeddings for $p\in P_0$, if they are defined, are the isomorphisms needed for $\G$'s categoricity. We will then prove key observations about the construction before stating the main verification lemma.

    \begin{lemma}\label{order-preserving map}
        Suppose that for all $p\in P$ and $e\in\omega$, the requirement $N_e^p$ is satisfied. Then, for all $p,q\in P$, we have that $p\leq q$ if and only if $D_p\leq_T D_q$.
    \end{lemma}
    \begin{proof}
    Assume that each $N_e^p$ requirement was satisfied and suppose that $p\leq q$. Since $p\leq q$, for all $r$ where $r\leq p$, we have that $r\leq q$ as well and so $D_p\leq_T D_q$. If $p\not\leq q$, then since
    \[
    A_p\not\leq_T\bigoplus_{t\neq p} A_t,
    \]
    it immediately follows that $A_p\not\leq_T\bigoplus\limits_{t\leq q} A_t$ and hence $D_p\not\leq_T D_q$.
    \end{proof}

    \begin{lemma}\label{Observation 3}
    If $f:\G\to\G$ is an embedding of $\G$ into itself, then $f$ is an isomorphism (and is, in fact, the identity map).
    \end{lemma}
    \begin{proof}
    Let $f:\G\to\G$ be an embedding. Since embeddings preserve loops and only the root nodes $a_m$ are contained in more than one loop, $f$ must map root nodes to root nodes. Furthermore, since only the root nodes $a_{2n}$ and $a_{2n+1}$ can have $(5n+1)$-loops, we can only have that $f(a_{2n})=a_{2n}$ or $f(a_{2n})=a_{2n+1}$. However, the only situation in which $a_{2n+1}$ has a $(5n+1)$-loop is when we attached a $(5n+3)$-loop to $a_{2n}$ but \textit{not} to $a_{2n+1}$. Thus, $f(a_{2n})=a_{2n}$ and similarly, $f(a_{2n+1})=a_{2n+1}$. Since $\G$ is a directed graph, it must map the loops attached to $a_{2n}$ and to $a_{2n+1}$ identically onto themselves.
    \end{proof}

    \begin{lemma}\label{Observation 4}
    If $\mathcal{M}_e\cong\G$ for a computable directed graph $\mathcal{M}_e$ and $f_e:\G\to\mathcal{M}_e$ is an embedding defined on all of $\G$, then $f_e$ is an isomorphism.
    \end{lemma}
    \begin{proof}
    This follows immediately from Lemma \ref{Observation 3}.
    \end{proof}

    By Lemma \ref{Observation 4}, if $\mathcal{M}_e\cong\G$ or $\G\cong\mathcal{M}_i^{D_p}$, then there is a unique isomorphism from $\G$ to $\mathcal{M}_e$ and from $\G$ to $\mathcal{M}_i^{D_p}$. We refer to the image of the $2n$th and $(2n+1)$st components of $\G$ in $\mathcal{M}_e^{D_p}$ as the \textbf{true copies} of these components in $\mathcal{M}_i^{D_p}$. We now prove several auxiliary lemmas about the construction.

    \begin{lemma}\label{true copies remain}
        If $\mathcal{G}\cong\mathcal{M}_i^{D_p}$, then for each $n$, there is an $s$ such that for all $t\geq s$, the true copies of the $2n$th and $(2n+1)$st components of $\G$ in $\mathcal{M}_i^{D_p}$ are the oldest and lexicographically least isomorphic copies in $\mathcal{M}_i^{D_p}[t]$ of these components.
    \end{lemma}
   \begin{proof}
       Let $u$ be the maximum $D_p$-use for the edges in the true copies of these components in $\mathcal{M}_i^{D_p}$ and let $s_0$ be such that $D_p\restriction(u+1)[s_0]=D_p\restriction(u+1)$. Because $D_p$ is c.e., the true components will be defined at every stage $s\geq s_0$. There may be a finite number of older fake copies of these components, but they will disappear as numbers enter $D_p$, and so for a large enough $s\geq s_0$, the true copies will be the oldest and lexicographically least in $\mathcal{M}_i^{D_p}[s]$.
   \end{proof}

\begin{lemma}\label{N restraint lemma}
    Let $\alpha$ be an $N_e^p$-strategy that enumerates $x_\alpha$ into $A_p$ at stage $s$. Unless $\alpha$ is initialized, no number below $\use(\Phi_e^{\widehat{D_p}}(x_\alpha)[s])$ is enumerated into $\widehat{D_p}$ after stage $s$.
\end{lemma}
\begin{proof}
    After $\alpha$ enumerates $x_\alpha$ into $A_p$ at stage $s$, all strategies extending $\alpha^\frown\<s\>$ will define new large parameters greater than $\use(\Phi_e^{\widehat{D_p}}(x_\alpha)[s])$. In particular, if there is an $R$-strategy $\beta\supseteq\alpha^\frown\<s\>$, it will define its parameter $n_\beta$ to be large. In the event that it issues a challenge to all higher priority $S$ and $T$ strategies $\gamma^\frown\<\infty\>\subseteq\alpha^\frown\<s\>\subseteq\beta$ on the $2n_\beta$th and $(2n_\beta+1)$st components, the use $u_{\gamma,n_\beta}$, if it exists, was chosen to be large (and so $u_{\gamma,n_\beta}>n_\beta$) on the mentioned components. Thus, $u_{\gamma,n_\beta}>\use(\Phi_e^{\widehat{D_p}}(x_\alpha)[s])$ also.
    
    The only strategies which have parameters smaller than $\use(\Phi_e^{\widehat{D_p}}(x_\alpha)[s])$ are to the left of $\alpha$ on the tree of strategies or are $R$- or $N$-strategies $\beta$ such that $\beta\subset\alpha$. When $\beta$ enumerates a number into their assigned c.e.\ set, then it will take outcome $s$ and initialize $\alpha$.
\end{proof}

   \begin{lemma}\label{Observation 1}
    An $S_e$-strategy or a $T_i^p$-strategy can be challenged by at most one $R$-strategy at any given stage.
    \end{lemma}
    \begin{proof}
    Let $\alpha$ be an $S_e$-strategy (or $T_i^p$-strategy) and suppose that there exists some $\beta$ such that $\alpha^\frown\<\infty\>\subseteq\beta$ and $\beta$ is an $R$-strategy that challenges $\alpha$. If $\beta$ challenges $\alpha$ at a stage $s$, then $\beta$ takes the $w_1$ outcome for the first time. The strategies extending $\beta^\frown\<w_1\>$ will choose witnesses at stage $s$, and in particular, none will challenge $\alpha$. So, at most one strategy will challenge $\alpha$ at stage $s$. Until $\alpha$ is able to match the newly added loops in $\G$ to meet the challenge, $\alpha$ will take outcome $w_{n_\alpha}$. $R$-strategies $\gamma$ such that $\gamma\supseteq\alpha^\frown\<w_{n_\alpha}\>$ will not challenge $\alpha$ since $w_{n_\alpha}\neq\infty$.
    \end{proof}

    \begin{lemma}\label{Observation 2}
    Suppose $\alpha$ is an $R_e^q$-strategy that is never initialized after stage $s$. Then $\alpha$ can only challenge higher priority $S$-strategies and $T$-strategies at most once after stage $s$.
    \end{lemma}
    \begin{proof}
    Suppose $\alpha$ is an $R_e^q$-strategy that is never initialized after stage $s$ and suppose it challenges all $S$-strategies and $T$-strategies $\beta$ such that $\beta^\frown\<\infty\>\subseteq\alpha$. Because $\alpha$ is never initialized again, if we ever return to $\alpha$, it will take the $s$ outcome as it can now diagonalize, and will continue to take the $s$ outcome at all subsequent $\alpha$-stages. If we do not return to $\alpha$, $\alpha$ will not be able to challenge any higher priority $S$-strategies or $T$-strategies after stage $s$ since it will never be eligible to act again.
    \end{proof}

    \begin{lemma}\label{enumeration per stage}
        At most one strategy $\alpha$ enumerates numbers at any stage.
    \end{lemma}
    \begin{proof}
        Suppose numbers are enumerated at a stage $s$ and $\alpha$ is the highest priority strategy which enumerates a number. $\alpha$ must either be for an $R_e^q$ or an $N_e^p$ requirement. If $\alpha$ is an $N_e^p$-strategy, it will take the $s$ outcome for the first time, and if $\alpha$ is an $R_e^q$-strategy, it will either take the $w_1$ outcome or the $s$ outcome for the first time. In either case, the remaining strategies which act at stage $s$ will act by simply defining their parameters and taking the $w_0$ outcome. Therefore, $\alpha$ is the only strategy to enumerate a number at stage $s$.
    \end{proof}

    \begin{lemma}\label{enumerating number less than T-uses}
        Let $\alpha$ be a $T_i^p$-strategy that defines $g_\alpha^{D_p}[s_m]$ on the $2m$th and $(2m+1)$st components of $\G$ at stage $s_m$. Until $\alpha$ is initialized (if ever), only strategies $\beta$ such that $\alpha^\frown\<\infty\>\subseteq\beta$ can enumerate a number $n\leq u_{\alpha,m}$ at any stage $t\geq s_m$.
    \end{lemma}
    \begin{proof}
        Let $\alpha$ be such a $T_i^p$-strategy. After defining $g_\alpha^{D_p}[s_m]$ on the $2m$th and $(2m+1)$st components of $\G$ at stage $s_m$, it takes the $\infty$ outcome. Hence, all strategies extending $\alpha^\frown\<w_k\>$ for some $k$ or to the right of $\alpha$ are initialized. These strategies will choose new parameters larger than $u_{\alpha,m}$ when they are next eligible to act. They are also the only strategies not extending $\alpha^\frown\<\infty\>$ which can enumerate numbers without initializing $\alpha$, so it suffices to show that they will not enumerate numbers below $u_{\alpha,m}$. Let $\beta$ be such a strategy.

        If $\beta$ is an $N_{e'}^{p'}$-strategy, then it can only enumerate its parameter $x_\beta$ into $A_{p'}$, and it was chosen such that $x_\beta>u_{\alpha,m}$. If $\beta$ is an $R_e^q$-strategy then it enumerates two types of numbers: $v_\beta$ and $u_{\gamma,n_\beta}$ for any $T_{i'}^{p'}$-strategy such that $\gamma^\frown\<\infty\>\subseteq\beta$ and $p'<q$. Since $v_\beta$ will be chosen large after stage $s_m$, we have that $v_\beta>u_{\alpha,m}$. Since $\beta\supseteq\alpha^\frown\<w_m\>$, we have that $n_\beta>u_{\alpha,m}$ because $n_\beta$ was chosen to be large after $u_{\alpha,m}$ was defined. In particular, when $n_\beta$ is chosen, $\gamma$ cannot have defined $u_{\gamma,n_\beta}$ since $n_\beta$ is large, so when $\gamma$ defines $u_{\gamma,n_\beta}$ later, it must satisfy $u_{\gamma,n_\beta}>n_\beta$. Hence $u_{\gamma,n_\beta}>u_{\alpha,m}$.
    \end{proof}

    \begin{lemma}\label{T-maps remain}
        Let $\alpha$ be a $T_i^p$-strategy that takes a $w_k$ outcome at a stage $s$. Let $t$ be the next $\alpha$-stage. Unless $\alpha$ has been initialized, $D_p[t]\restriction \<u_{\alpha,m},p\>=D_p[s]\restriction \<u_{\alpha,m},p\>$ for all $m<n_\alpha$, and so $g_\alpha^{D_p}[t]$ remains defined on the $2m$th and $(2m+1)$st components of $\G$ for all $m<n_\alpha$. 
    \end{lemma}
    \begin{proof}
        This follows immediately from Lemma \ref{enumerating number less than T-uses}.
    \end{proof}

    \begin{lemma}\label{T-uses are greater than uses for loops}
        Let $\alpha$ be a $T_i^p$-strategy. If $\alpha$ defines $g_\alpha^{D_p}$ on the $2m$th and $(2m+1)$st components of $\G$ at stage $s$, then $l_m[s]<u_{\alpha,m}[s]$ where $l_m[s]$ is the max use of the edges in the $(5m+1)$- and $(5m+2)$-loops in the copies of the $2m$th and $(2m+1)$st components of $\G$ in $\mathcal{M}_i^{D_p}$ and $\<u_{\alpha,m}[s],p\>$ is the $D_p$-use of $g_\alpha^{D_p}[s]$ on these components. Furthermore, for all $\alpha$-stages $t>s$, we have $u_{\alpha,m}[t]\geq u_{\alpha,m}[s]>l_m[s]=l_m[t]$ unless these components in $\mathcal{M}_i^{D_p}$ are injured or $\alpha$ is initialized.
    \end{lemma}
    \begin{proof}
        When $\alpha$ initially defines $g_\alpha^{D_p}$ on those components in \textbf{Case} $\mathbf{2}$ or \textbf{Case} $\mathbf{3}$ of its strategy, it chooses $u_{\alpha,m}[s]$ large, and so $u_{\alpha,m}[s]>l_m[s]$.

        Consider an $\alpha$-stage $t>s$ and assume $\alpha$ has not been initialized and the components in $\mathcal{M}_i^{D_p}$ remain intact. Since the $D_p$-use on the edges in the components remains the same, we have $l_m[t]=l_m[s]$, and so $(D_p\restriction l_m[s])[t]=(D_p\restriction l_m[s])[s]$. In particular, any update of the value of $n_\alpha$ in \textbf{Case} $\mathbf{2}$ or in \textbf{Step} $\mathbf{1}$ of \textbf{Case} $\mathbf{3}$ of the $T_i^p$-strategy does not cause $n_\alpha$ to fall below $m$. Therefore, $u_{\alpha,m}[t]$ either has the same value as in the previous $\alpha$-stage, or is updated by the preliminary action of \textbf{Case} $\mathbf{2}$ (and so is chosen large), or is redefined in \textbf{Step} $\mathbf{2}$ of \textbf{Case} $\mathbf{3}$ (to its value at the previous $\alpha$-stage). In all cases, $u_{\alpha,m}[t]\geq u_{\alpha,m}[s]$. 
    \end{proof}

    \begin{lemma}\label{restraint by R stays}
        Let $\alpha$ be an $R_e^q$-strategy that acts in \textbf{Case $\mathbf{2}$} at stage $s$ by defining $v_\alpha$. Let $t>s$ be the next $\alpha$-stage and assume $\alpha$ is not initialized before $t$.
        \begin{itemize}
            % \item[(1)] $D_q[t]\restriction \<v_\alpha,q\>=D_q[s]\restriction \<v_\alpha,q\>$, and
            \item[(1)] At stage $t$, for all $S$- and $T$-strategies $\beta$ where $\beta^\frown\<\infty\>\subseteq\alpha$, $f_\beta$ and $g_\beta^{D_p}[t]$ is defined on the $2n_\alpha$th and $(2n_\alpha+1)$st components of $\G$. Furthermore, for $T_i^p$-strategies $\beta$, the minimum use of the computations $g_\beta^{D_p}[t]$ on the $(5n_\alpha+3)$- and $(5n_\alpha+4)$-loops is greater than $v_\alpha$.
            \item[(2)] If $q<p$, then the $D_p$-use for $g_\beta^{D_p}[t]$ on $a_{2n_\alpha}$, $a_{2n_\alpha+1}$, the $(5n_\alpha+1)$-loops, $(5n_\alpha+2)$-loops, and the $2$-loops in $\G$ is greater than $v_\alpha$.
        \end{itemize}
    \end{lemma}
    \begin{proof}
        For ($1$), since $\alpha$ took the $w_1$ outcome at the last $\alpha$-stage $s$ and was not initialized after, it is now in \textbf{Case $\mathbf{3}$} at stage $t$. So in particular, it must be the case that $\alpha$ saw that for all $S$- and $T$-strategies $\beta$ where $\beta^\frown\<\infty\>\subseteq\alpha$ that their parameters $n_\beta$ have exceeded $n_\alpha$. In particular, if $\beta$ is a $T_i^p$-strategy, it extended its map $g_\beta^{D_p}[s]$ on the $(5n_\alpha+3)$- and $(5n_\alpha+4)$-loops in $\G$ with a large use $u_{\beta,n_\alpha}$. For each such $\beta$, we have that $u_{\beta,n_\alpha}>v_\alpha$. Furthermore, we claim that $\beta$ cannot be challenged by another $R$-strategy before stage $t$. Suppose $\gamma$ challenges $\beta$ after $\beta$ meets $\alpha$'s challenge. If $\beta^\frown\<\infty\>\subseteq\gamma\subset\alpha$, then $\gamma$ takes outcome $w_1$ when it challenges $\beta$, initializing $\alpha$. If $\alpha\subseteq\gamma$, then $\gamma$ cannot act until after stage $t$.

        For ($2$), if $\beta$ was a $T_i^p$-strategy where $q<p$ and $n_\beta>n_\alpha$, then $\alpha$ enumerated $u_{\beta,n_\alpha}$ into $A_p$, causing the map $g_\beta^{D_p}[s]$ to now be undefined on the entirety of the $2n_\alpha$th and $(2n_\alpha+1)$st components of $\G$. At stage $t$ when $\alpha$ is eligible to act again, we have that $\beta$ must have recovered its map on $a_{2n_\alpha}$, $a_{2n_\alpha+1}$, the $(5n_\alpha+1)$-loops, $(5n_\alpha+2)$-loops, and the $2$-loops in $\G$ with a new large use $u_{\beta,n_\alpha}>v_\alpha$.

        On the other hand, if $n_\beta\leq n_\alpha$, when $\alpha$ challenged $\beta$, then $g_\beta^{D_p}$ was not yet defined on any of the loops in the $2n\alpha$th and $(2n_\alpha+1)$st components. Therefore, when $\beta$ defines $g_\beta^{D_p}$ on these components, it uses a large use on every loop.
    \end{proof}

    \begin{lemma}\label{loops in B remain}
        Let $\alpha$ be an $R_e^q$-strategy that acts in \textbf{Case} $\mathbf{2}$ at stage $s$ by defining $v_\alpha$ and challenging higher priority $S$ and $T$-strategies. If $t$ is the next $\alpha$-stage and $\alpha$ has not been initialized, then $D_q[t]\restriction\<v_\alpha,q\>+1=D_q[s]\restriction\<v_\alpha,q\>+1$. In particular, all of the loops in the $2n_\alpha$th and $(2n_\alpha+1)$st components of $\B_q$ remain intact.
    \end{lemma}
    \begin{proof}
        At stage $s$, $\alpha$ takes the $w_1$ outcome. By the proof of Lemma \ref{enumeration per stage}, no other strategy enumerates numbers at stage $s$. Since the strategies extending $\alpha^\frown\<w_1\>$ and to the right of $\alpha$ choose new large parameters, none can enumerate a number below $\<v_\alpha,q\>$ into $D_q$. If a strategy $\beta\subseteq\alpha$ enumerates a number, the path moves left, initializing $\alpha$. Therefore, unless $\alpha$ is initialized, no number below $\<v_\alpha,q\>$ can enter $D_q$ before stage $t$.
    \end{proof}

    \begin{lemma}\label{after R success}
        Let $\alpha$ be an $R_e^q$-strategy that acts in \textbf{Case} $\mathbf{3}$ at stage $s$ and takes the $s$ outcome. Unless $\alpha$ is initialized, no number below the uses of the loops in the $2n_\alpha$th and $(2n_\alpha+1)$st components of $\B_q$ or below $m_q$ is enumerated into $D_q$ after stage $s$.
    \end{lemma}
    \begin{proof}
        The proof of this lemma is almost identical to the proof of Lemma \ref{N restraint lemma}.
    \end{proof}

We now state and prove the verification lemma for our construction.

    \begin{lemma}[Main Verification Lemma]\label{verification lemma}
    Let $\pi=\liminf_s \pi_s$ be the true path of the construction, where $\pi_s$ denotes the current true path at stage $s$ of the construction. Let $\alpha\subset p$.
    \begin{itemize}
        \item[(1)] If $\alpha$ is an $N_e^p$-strategy, then there is an outcome $o$ and an $\alpha$-stage $t_\alpha$ such that for all $\alpha$-stages $s\geq t_\alpha$, $\alpha$ takes outcome $o$ where $o$ ranges over $\{s,w_0\}$.
        \item[(2)] If $\alpha$ is an $S_e$-strategy, then either $\alpha$ takes outcome $\infty$ infinitely often or there is an outcome $w_n$ and a stage $\hat{t}$ such that for all $\alpha$-stages $s>\hat{t}$, $\alpha$ takes outcome $w_n$. If $\G\cong\mathcal{M}_e$, then $\alpha$ takes the $\infty$ outcome infinitely often and $\alpha$ defines a partial embedding $f_\alpha:\G\to\mathcal{M}_e$ which can be extended to a computable isomorphism $\hat{f}_\alpha:\G\to\mathcal{M}_e$.
        \item[(3)] Let $\alpha$ be a $T_i^p$-strategy. If $\G\cong\mathcal{M}_i^{D_p}$, then $\alpha$ takes the $\infty$ outcome infinitely often, and $\alpha$ defines a partial embedding $g_\alpha^{D_p}:\G\to\mathcal{M}_i^{D_p}$ which can be extended to a $D_p$-computable isomorphism $\hat{g}_\alpha^{D_p}:\G\to\mathcal{M}_i^{D_p}$.
        \item[(4)] If $\alpha$ is an $R_e^q$-strategy, then there is an outcome $o$ and an $\alpha$-stage $t_\alpha$ such that for all $\alpha$-stages $s\geq t_\alpha$, $\alpha$ takes outcome $o$ where $o$ ranges over $\{s, w_1, w_0\}$.
    \end{itemize}
    In addition, $\alpha$ satisfies its assigned requirement.
    \end{lemma}

    \begin{proof}%start of proof environment for verification lemma
    We first prove ($1$). Let $\alpha\subseteq \pi$ be an $N_e^p$-strategy and let $s_0$ be the least stage such that for all $s\geq s_0$, $\alpha\leq_L \pi_s$. Let $x_\alpha$ be its parameter at stage $s_0$. Suppose that at every $\alpha$-stage $s\geq s_0$, either $\Phi_e^{\widehat{D_p}}(x_\alpha)[s]\uparrow$ or $\Phi_e^{\widehat{D_p}}(x_\alpha)[s]\downarrow\neq 0$. Then, $\alpha$ takes the $w_0$ outcome at every $\alpha$-stage $s\geq s_0$. Furthermore, $\Phi_e^{\widehat{D_p}}(x_\alpha)$ either diverges or converges to a number other than $0$. Since $x_\alpha\not\in A_p$, $\alpha$ has met $N_e^p$.

    Otherwise, there is some $\alpha$-stage $t>s_0$ where $\Phi_e^{\widehat{D_p}}(x_\alpha)[t]\downarrow=0$. At stage $t$, $\alpha$ enumerates $x_\alpha$ into $A_p$ and takes the $s$ outcome. Since $\alpha$ is never initialized, it takes the $s$ outcome at every $\alpha$-stage $s\geq t$. Furthermore, by Lemma \ref{N restraint lemma} we have that
    \[
    \Phi_e^{\widehat{D_p}}(x_\alpha)=\Phi_e^{\widehat{D_p}}(x_\alpha)[t]=0\neq 1=A_p(x_\alpha),
    \]
    and so $N_e^p$ is satisfied.

    For ($2$), let $\alpha\subseteq \pi$ be an $S_e$-strategy and let $s_0$ be the least stage such that for all $s\geq s_0$, $\alpha\leq_L \pi_s$. Suppose that $\alpha$ only takes the $\infty$ outcome finitely often. Fix an $\alpha$-stage $s_1>s_0$ such that $\alpha$ does not take the $\infty$ outcome at any $\alpha$-stage $s\geq s_1$. Suppose that $\alpha$ takes the $w_n$ outcome at stage $s_1$. There are two cases to consider.

    If $\alpha$ is not challenged at stage $s_1$, then $\alpha$ acts as in \textbf{Case} $\mathbf{3}$ of the $S_e$-strategy. Since $\alpha$ cannot be challenged by a requirement extending $\alpha^\frown\<w_n\>$, $\alpha$ remains in \textbf{Case} $\mathbf{3}$ at future $\alpha$-stages unless it finds copies of the $2n$th and $(2n+1)$st components of $\G$ in $\mathcal{M}_e$. However, if it finds these copies, it would take the $\infty$ outcome, and so $\alpha$ must not ever find these copies. Hence, $\alpha$ takes the $w_n$ outcome at every $\alpha$-stage $s\geq s_1$. Moreover, $\mathcal{M}_e$ does not contain copies of the $2n$th and $(2n+1)$st components of $\G$ and so $\mathcal{M}_e$ is not isomorphic to $\G$. 

    If $\alpha$ is challenged at stage $s_1$, then $\alpha$ acts as in \textbf{Case} $\mathbf{2}$ of the $S_e$-strategy. By a similar argument to the one above, $\alpha$ can never meet this challenge by finding images for the new loops in the $2n$th and $(2n+1)$st components of $\G$. Therefore, $\mathcal{M}_e\not\cong\G$ and $\alpha$ takes the $w_n$ outcome for all $\alpha$-stages $s\geq s_1$.

    From these two cases, it follows that if $\G\cong\mathcal{M}_e$, then $\alpha$ must take the $\infty$ outcome infinitely often. Let $n_\alpha[s]$ denote the value of $n_\alpha$ at the end of stage $s$ (i.e., $n_\alpha$ could be possibly redefined by a challenging strategy during stage $s$). We claim that the value of $n_\alpha[s]$ goes to infinity as $s$ goes to infinity. Suppose that $\alpha$ takes the $\infty$ outcome at a stage $s>s_0$. Either $n_\alpha[s]=n_\alpha[s-1]+1$ or $n_\alpha[s]=n_\beta[s]$ for some $R$-strategy $\beta$ where $\alpha^\frown\<\infty\>\subseteq\beta$ and $n_\beta[s]\leq n_\alpha[s-1]$. There can only be finitely many such $\beta$ with $n_\beta[s]\leq n_\alpha[s-1]$. Once those strategies have challenged $\alpha$ (if ever), the value of $n_\alpha$ cannot drop below $n_\alpha[s-1]$ again. As $\alpha$ takes the $\infty$ outcome infinitely often, it follows that $n_\alpha[s]$ goes to infinity as $s$ increases.

    Let $f_\alpha=\bigcup\limits_{s\geq s_0} f_\alpha[s]$ be the limit of the partial $f_\alpha[s]$ embeddings for $s\geq s_0$. By Lemma \ref{Observation 4}, it remains to show that $f_\alpha$ can be computably extended to an embedding $\hat{f}_\alpha$ which is defined on all of $\G$. Note that no strategy $\beta\subseteq\alpha$ can add loops to $\G$ after stage $s_0$ or else the current true path would move to the left of $\alpha$ in the tree. Since $n_\alpha[s]\to\infty$ as $s\to\infty$, for each $k$, there is an $\alpha$-stage $s_k$ at which the $n_\alpha$ parameter starts with value $k$ and $\alpha$ takes the $\infty$ outcome. At this stage, $\alpha$ found a copy of the $2k$th and $(2k+1)$st components of $\G$ (which consist of at least the initial set of loops) in $\mathcal{M}_e$ and defined $f_\alpha[s_k]$ on these loops. Therefore, $f_\alpha$ is defined on all of the initial loops attached to each $a_{2k}$ and $a_{2k+1}$. 

    Only $R$-strategies $\beta$ such that $\alpha^\frown\<\infty\>\subseteq\beta$ can add loops to components on which $f_\alpha$ is already defined. If such a strategy $\beta$ adds loops as in \textbf{Case} $\mathbf{3}$ of its action, then it challenges $\alpha$ to find copies of these new loops. Since $\alpha$ takes the $\infty$ outcome infinitely often, it meets this challenge and extends $f_\alpha$ to be defined on the loops created by $\beta$. Then, $\beta$ adds the homogenizing loops as in \textbf{Case} $\mathbf{4}$ of its action. If $\G\cong\mathcal{M}_e$, then these homogenizing loops will have copies in $\mathcal{M}_e$. Suppose the $2n$th and $(2n+1)$st components of $\G$ have homogenizing loops, then because $\G\cong\mathcal{M}_e$, then only the nodes $f_\alpha(a_{2n})$ and $f_\alpha(a_{2n+1})$ in $\mathcal{M}_e$ will have copies of the respective homogenizing loops. So, we can computably extend $f_\alpha$ to $\hat{f}_\alpha$ by mapping the homogenizing loops to these copies, and thus $\hat{f}_\alpha$ is the computable isomorphism which satisfies the $S_e$ requirement.

    For ($3$), let $\alpha\subseteq \pi$ be a $T_i^p$-strategy and let $s_0$ be the least stage such that $\alpha\leq_L \pi_s$ for all $s\geq s_0$. If $\G\not\cong\mathcal{M}_i^{D_p}$, then we satisfy the $T_i^p$ requirement trivially. So, suppose $\G\cong\mathcal{M}_i^{D_p}$.

    We claim that $\alpha$ takes the $\infty$ outcome infinitely often and that $n_\alpha[s]\to\infty$. By Lemma \ref{true copies remain}, the true copies of each pair of components will eventually appear in $\mathcal{M}_i^{D_p}$ and become the oldest copies of these components. If $g_\alpha^{D_p}$ maps the $2m$th and $(2m+1)$st components of $\G$ to fake copies in $\mathcal{M}_i^{D_p}$, then by Lemma \ref{T-uses are greater than uses for loops}, when a number less than $l_m[s]$ enters $D_p$ to remove the fake copies, then the map $g_\alpha^{D_p}[s]$ also becomes undefined on those components. Therefore, for each $n$, $\alpha$ will eventually define $g_\alpha^{D_p}[s]$ correctly on the $2n$th and $(2n+1)$st components of $\G$, mapping them to the true copies in $\mathcal{M}_i^{D_p}$.

    For the same reason, $\alpha$ will also meet each challenge after $s_0$ by an $R$-strategy extending $\alpha^\frown\<\infty\>$. It follows that $n_\alpha[s]\to\infty$ as $s\to\infty$ and that $g_\alpha^{D_p}=\bigcup\limits_{s\geq s_0} g_\alpha^{D_p}[s]$ will correctly map all loops in $\G$ into $\mathcal{M}_i^{D_p}$ except the homogenizing loops added in \textbf{Case} $\mathbf{3}$ of an $R$-strategy. 
    
    It remains to show that we can extend $g_\alpha^{D_p}$ in a $D_p$-computable way to an embedding $\hat{g}_\alpha^{D_p}$ defined on all of $\G$. Using the $D_p$ oracle, we can tell when $g_\alpha^{D_p}[s]$ has correctly defined the original $(5m+1)$- and $(5m+2)$-loops from $\G$ into $\mathcal{M}_i^{D_p}$, as well as the $(5m+3)$- and $(5m+4)$-loops added (if ever) to $\G$ by an $R$-requirement. The $D_p$ oracle will then tell us when the correct homogenizing loops show up in $\mathcal{M}_i^{D_p}$ (assuming that they were added to $\G$), so that we can extend $g_\alpha^{D_p}$ in a $D_p$-computable manner.

    For ($4$), suppose $\alpha\subset \pi$ is an $R_e^q$-strategy and let $n_\alpha=n$ be its parameter. Let $s_0$ be the least stage such that $\alpha\leq_L \pi_s$ for all $s\geq s_0$. If $\alpha$ remains in the first part of \textbf{Case} $\mathbf{2}$ from its description for all $\alpha$-stages $s\geq s_0$, then $R_e^q$ is trivially satisfied because $\Phi_e^{D_q}$ is not an isomorphism between $\G$ and $\B_q$, and $\alpha$ takes the $w_0$ outcome cofinitely often. 

    Otherwise, there is an $\alpha$-stage $s_1>s_0$ such that $\Phi_e^{D_q}[s_1]$ maps the $2n$th and $(2n+1)$st components of $\G$ isomorphically into $\B_q$. Then, $\alpha$ carries out all actions described in the second part of \textbf{Case} $\mathbf{2}$. In particular, it defines its target number $v_\alpha$ \textit{after} enumerating all uses $u_{\gamma,n}$ (if they exist) for any $T_i^p$-strategy $\gamma$ of higher priority with $q<p$ in $P$. $\alpha$'s challenge will eventually be met by all $\beta$ such that $\beta^\frown\<\infty\>\subset\alpha$ since $\beta^\frown\<\infty\>\subset \pi$, and so let $s_2>s_1$ be the next $\alpha$-stage. By Lemma \ref{loops in B remain}, $D_q[s_2]\restriction\<v_\alpha,q\>+1=D_q[s_1]\restriction\<v_\alpha,q\>+1$, so the loops added to $\B_q$ at stage $s_1$ remain intact. At stage $s_2$, $\alpha$ enumerates $v_\alpha$ into $A_q$, moves the $(5n+3)$- and $(5n+4)$-loops in $\B_q$, and takes the $s$ outcome at the end of this stage and at every future $\alpha$-stage. By Lemma \ref{restraint by R stays}, for every $T_i^p$-strategy $\beta$ with $\beta^\frown\<\infty\>\subseteq\alpha$, we have that
    \begin{itemize}
        \item $g_\beta^{D_p}[s_2+1]$ is now undefined on the $(5n_\alpha+3)$- and $(5n_\alpha+4)$-loops in the $2n_\alpha$th and $(2n_\alpha+1)$st components of $\G$, and
        \item if $q<p$, then $g_\beta^{D_p}[s_2+1]$ is now undefined on the entirety of the $2n_\alpha$th and $(2n_\alpha+1)$st components of $\G$.
    \end{itemize}
    
    Since $\alpha\subset \pi$, $\alpha$ will never get initialized again. By Lemmas \ref{loops in B remain} and \ref{after R success}, we preserve $D_q\restriction m_\alpha$. Recall that $m_\alpha$ is the use of the computation of $\Phi_e^{D_q}[s_1]$ where $\Phi_e^{D_q}[s_1](a_{2n})=b_{2n}^q$ and $\Phi_e^{D_q}[s_1](a_{2n+1})=b_{2n+1}^q$. So we have that $\Phi_e^{D_q}=\Phi_e^{D_q}[s_1]$, and so $\Phi_e^{D_q}(a_{2n})=b_{2n}^q$ and $\Phi_e^{D_q}(a_{2n+1})=b_{2n+1}^q$. However, $a_{2n}$ is connected to a cycle of length $5n+3$ whereas $b_{2n}$ is connected to a cycle of length $5n+4$, and so $\Phi_e^{D_q}$ cannot be a $D_q$-computable isomorphism between $\G$ and $\B_q$.
    \end{proof}%end of proof environment for verification lemma

\section{Conclusion}
The techniques used in the proof of Theorem \ref{main theorem} to make $\G$ computably categorical or not relative to a degree are compatible with techniques to create minimal pairs of c.e.\ degrees. In fact, we can show that we can embed the four element diamond lattice into the c.e.\ degrees in the following way.

\begin{theorem}
    There exists a computable computably categorical directed graph $\G$ and c.e.\ sets $X_0$ and $X_1$ such that
    \begin{itemize}
        \item[(1)] $X_0$ and $X_1$ form a minimal pair,
        \item[(2)] $\G$ is not computably categorical relative to $X_0$ and to $X_1$, and
        \item[(3)] $\G$ is computably categorical relative to $X_0\oplus X_1$.
    \end{itemize}
\end{theorem}

A natural question is whether we can embed bigger lattices into the c.e.\ degrees where elements of the lattice are labelled as either ``c.c.'' or ``not c.c.'' like in the poset case. We restrict to distributive lattices since those are embeddable into the c.e.\ degrees.

\begin{question}\label{lattices}
    Let $L=(L,\land,\lor)$ be a computable distributive lattice and suppose we have a computable partition $L=L_0\sqcup L_1$. Does there exist a computable computably categorical structure $\mathcal{A}$ and an embedding $h$ of $L$ into the c.e.\ degrees where $\mathcal{A}$ is computably categorical relative to each degree in $h(L_0)$ and is not computably categorical relative to each degree in $h(L_1)$?
\end{question}

We end the paper by first discussing a restrictive case where our techniques do not work. Before we state this restrictive case, we observe that we can obtain a version of Theorem \ref{main theorem} where our graph $\mathcal{G}$ is not computably categorical using largely the same construction. Franklin and Solomon in \cite{franklinsolomon2014} showed that every $2$-generic degree $\mathbf{d}$ is low for isomorphism, that is, for every pair of computable structures $\mathcal{A}$ and $\mathcal{B}$, $\mathcal{A}$ is $\mathbf{d}$-computably isomorphic to $\mathcal{B}$ if and only if $\mathcal{A}$ is computably isomorphic to $\mathcal{B}$. So for every $2$-generic $\mathbf{d}$, there exists \textit{no} computable structure $\mathcal{A}$ where $\mathcal{A}$ is not computably categorical but is computably categorical relative to $\mathbf{d}$. This is also optimal in the generic degrees, meaning that for a $1$-generic degree $\mathbf{c}$, we can apply techniques from this paper to build a computable structure that can change its behavior from being not computably categorical to being computably categorical relative to $\mathbf{c}$. A proof of this fact will appear in a sequel paper by Villano \cite{villano2025}.

It is still unknown how categoricity relative to a degree behaves above $\mathbf{0}'$ and strictly below $\mathbf{0}''$. A concrete question in this direction is the following.

\begin{question}
    For degrees in the the interval $[\mathbf{0}',\mathbf{0}'')$, is computable categoricity relative to a degree nonmonotonic like in the c.e.\ degrees?
\end{question}

It is likely that categoricity relative to a degree also has nonmonotonic behavior in the interval $[\mathbf{0}',\mathbf{0}'')$, though a construction to witness this may require a $\mathbf{0}'''$-priority argument. 

\printbibliography

\end{document}